\documentclass[12pt]{article}
\usepackage[margin=1.25in]{geometry}
\usepackage{pgfplots}
\usepackage{subfig}
\usepackage{amsmath}
\usepackage{amssymb}
\usepackage{algorithmic}
\usepackage{algorithm}
\usepackage{mathtools}
\usepackage{color}
\usepackage[normalem]{ulem}
\usepackage{setspace}
\usepackage{tikz}
\usetikzlibrary{arrows,
	arrows.meta,
	bending,calc,patterns,angles,quotes}
\usepackage{float}
\usepackage{hyperref}

\makeatletter
\providecommand*{\cupdot}{%
	\mathbin{%
		\mathpalette\@cupdot{}%
	}%
}
\newcommand*{\@cupdot}[2]{%
	\ooalign{%
		$\m@th#1\cup$\cr
		\hidewidth$\m@th#1\cdot$\hidewidth
	}%
}
\makeatother

\newtheorem{theorem}{Theorem}

\newtheorem{lemma}{Lemma}

\newtheorem{remark}{Remark}

\def\endpf{{\ \hfill\hbox{\vrule width1.0ex height1.0ex}\parfillskip 0pt
	}}
\newenvironment{proof}{\noindent{\bf Proof:}}{\endpf}
\newcounter{figurecounter}
\begin{document}
\setlength{\baselineskip}{20pt}

\title{Nonparametric inference from possibly unstable M/G/1 workload observations}

\author{Royi Jacobovic\footnote{School of Mathematical Sciences, Tel-Aviv University, Tel-Aviv, Israel, 6997800, E-mail: royijacobo@tauex.tau.ac.il.} \footnote{Jacobovic acknowledges the support of the Israel Science Foundation, Grant \#3739/24.} \and Binyamin Kobzantsev \footnote{School of Mathematical Sciences, Tel-Aviv University, Tel-Aviv, Israel, 6997800, E-mail: kobzantsev1@mail.tau.ac.il}}
\maketitle


\begin{abstract} Hansen and Pitts (2006) introduced the problem of nonparametric estimation of the service-time distribution of an M/G/1 queue observed through its workload process at discrete times $t=0,1,\ldots,n$. Despite its seemingly simple formulation, obtaining an estimator with sharp risk guarantees for this observation model has remained an open challenge for nearly two decades. 

In this paper, we construct such an estimator and prove that its $L^1$-risk is $\mathcal{O}\left(\frac{\log n}{\sqrt{n}}\right)$ as $n\to\infty$. Remarkably, this nearly parametric convergence rate is achieved without assuming stationarity, stability, or knowledge of the arrival rate. 

Our approach is based on a two-stage screening procedure that uncovers a hidden conditionally independent compound Poisson structure within the dependent workload observations. This probabilistic reduction transforms the original estimation problem into a classical decompounding problem, making it possible to leverage existing nonparametric estimation techniques despite the complex dependence induced by the reflected workload process. 

More broadly, we hope that the proposed screening methodology will provide a useful framework for statistical inference from dependent stochastic systems beyond the classical assumption of stability.\end{abstract}

Nonparametric estimation, Screening, Decompounding, Coupling, M/G/1 queue, Skorokhod reflection.
\smallskip


\section{Introduction} 
Recent advances in data collection technologies have generated a growing interest in statistical inference for queueing systems. Modern service systems routinely produce large amounts of operational data, creating new opportunities to estimate system characteristics directly from observations. However, the development of reliable inference procedures remains challenging due to the complex dependence structures inherent in queueing dynamics.

In this paper, we revisit one of the long-standing open problems in statistical inference for queueing systems, introduced by Hansen and Pitts \cite{Hansen2006}. They considered a statistician who observes the workload process of an M/G/1 queue at equally spaced time instants ($t=0,1,\ldots,n$) and aims to estimate the unknown service-time distribution $B(\cdot)$ at a prescribed point $w>0$. Actually, despite the apparent simplicity of this observation model, obtaining an estimator with sharp risk guarantees, has remained an open problem for nearly two decades.

Namely, the difficulty of solving this problem originates from the hidden dynamics between consecutive observation epochs. The arrival process, service completions, and possible visits of the workload process to zero are all unobserved. Consequently, the workload increments exhibit a complex dependence structure induced by the reflection mechanism and do not directly reveal the probabilistic characteristics of the underlying input process. This dependence has been the main obstacle to achieving sharp nonparametric estimation guarantees in this setting.

In this work, we resolve this problem by constructing a fully data-driven estimator $B_n(w)$ satisfying
\[ \mathbb E|B_n(w)-B(w)| = \mathcal O\!\left( \frac{\log n}{\sqrt n} \right), \qquad n\to\infty. \]
Remarkably, this nearly parametric convergence rate is achieved without assuming stationarity, stability, or knowledge of the arrival rate, requiring only mild smoothness assumptions on the service-time distribution $B(\cdot)$. This result provides a rare example of a statistical procedure for a \textit{possibly unstable} queueing system with rigorous nonparametric risk guarantees. We hope that it will motivate further research on inference procedures that do not rely on the classical assumption of stability.

The main methodological contribution of this paper is a two-stage screening methodology that reveals a hidden conditionally independent structure within the dependent workload observations. Rather than estimating directly from the original trajectory, we construct a data-driven screening procedure that extracts a subset of workload increments coinciding with increments of the underlying compound Poisson input process. The retained observations form a conditionally independent compound Poisson sample, thereby reducing the original inference problem to a classical decompounding problem. This probabilistic reduction allows us to define $B_n(w)$ by combining the proposed screening procedure with the Laplace-transform methodology of Den Boer and Mandjes \cite{Den Boer2017}. 

Beyond the specific queueing application, the methodology developed here suggests a broader probabilistic principle: strongly dependent observations may conceal a hidden conditionally independent structures that can be recovered through an appropriate screening procedure. In the present setting, this principle transforms inference from a reflected workload process into a classical decompounding problem. We expect that this perspective may prove useful in other stochastic systems where dependence arises through reflection, regeneration, or related mechanisms.\newline\newline
\noindent\textbf{Related literature.}
The central methodological contribution of the present paper is a probabilistic
reduction that converts inference from dependent workload observations into a
classical nonparametric decompounding problem. The existing literature relevant
to this contribution naturally falls into four directions: statistical inference
from M/G/1 workload observations, nonparametric decompounding, transform-based
inference for queueing models, and classical two-stage statistical procedures.

The estimation problem considered here was first formulated by Hansen and Pitts
\cite{Hansen2006}, who proposed an estimator based on the
Pollaczek--Khinchine transform identity together with empirical Laplace
transforms. Their pioneering work established that, under the stability
assumption, the service-time distribution can in principle be recovered from
discrete observations of the workload process, thereby initiating the study of
nonparametric inference from M/G/1 workload data.

Despite this important breakthrough, the estimation problem remains
considerably more challenging than classical transform-based inference
problems. The Skorokhod reflection mechanism induces strong temporal
dependence in the workload observations, so that the observed increments no
longer coincide with those of the underlying compound Poisson input process.
Consequently, standard decompounding techniques cannot be applied directly.
Moreover, the Hansen--Pitts framework is restricted to stable queueing systems,
whereas the present paper allows both stable and unstable M/G/1 queues.
Finally, although Hansen and Pitts demonstrated the feasibility of statistical
recovery, no estimator achieving a nearly parametric convergence rate has
previously been established for this observation scheme.
A complementary line of research was recently initiated by Ravner
\cite{Ravner2026}, who considers Poisson sampling of the workload process.
That framework relies on substantially stronger assumptions than those adopted
here, including stationarity of the workload process, knowledge of the arrival
rate, the strengthened stability condition
\[
\lambda\int_0^\infty x\,{\rm d}B(x)<1-\delta,
\]
for some $\delta\in(0,1/2)$, together with additional smoothness and moment
assumptions on the service-time distribution. Under these hypotheses, Ravner
obtains non-asymptotic risk bounds by combining Fourier inversion techniques
with structural results developed in Ravner, Boxma and Mandjes
\cite{Ravner2019}. The resulting convergence rate is of order
\[
n^{-\eta/(\eta+1)},
\]
where $\eta>0$ denotes the smoothness of the service-time distribution.

Our screening construction is conceptually related to an idea appearing in
Ravner, Boxma and Mandjes \cite[Section~5]{Ravner2019}, where screening is
employed in estimating the Lévy exponent of a spectrally positive Lévy-driven
queue observed through Poisson probing. Despite this similarity, the role of
screening is fundamentally different. In Ravner et al., screening serves
primarily to identify observations with favorable statistical properties,
thereby improving the stability of the estimation procedure. In contrast, the
proposed two-stage screening procedure performs a probabilistic reduction. It
identifies observation intervals during which the Skorokhod reflection is
inactive, thereby reducing the original dependent-data problem to inference
from an embedded conditionally independent compound Poisson sample.

Once this reduction has been achieved, the estimation problem falls within the
well-established framework of nonparametric decompounding. This allows us to
combine the proposed screening procedure with the Laplace-transform methodology
developed by Den Boer and Mandjes \cite{Den Boer2017}. Consequently, the
resulting estimator inherits the statistical properties of the underlying
decompounding procedure while requiring neither stationarity, stability, nor
knowledge of the arrival rate.

Nonparametric decompounding itself has become a mature area of statistical
inference. Important contributions include Buchmann and Grübel
\cite{Buchmann2003}, Van Es, Gugushvili and Spreij \cite{Van Es2007},
Gugushvili \cite{Gugushvili2012}, and Den Boer and Mandjes
\cite{Den Boer2017}. The present work establishes a direct connection between
this literature and statistical inference for reflected queueing systems by
showing that an appropriate screening procedure reduces the M/G/1 workload
observation problem to a classical decompounding problem.

Transform-based statistical methods have also proved highly successful in
queueing models possessing stronger independence properties. Representative
examples include inference for the M/G/$\infty$ queue by Pickands and Stine
\cite{Pickands1997}, Bingham and Pitts \cite{Bingham1999}, Blanghaps, Nov and
Weiss \cite{Blanghaps2013}, Goldenshluger
\cite{Goldenshluger2016,Goldenshluger2018}, and Goldenshluger and Jacobovic
\cite{Goldenshluger2024}. In these models, inference relies on covariance
structures or transform identities that are considerably easier to exploit than
those available for the reflected workload process of the M/G/1 queue.

The proposed two-stage screening procedure also bears a conceptual resemblance
to classical two-stage sampling procedures in mathematical statistics.
Stein's pioneering work \cite{Stein1945}, together with the ranking-and-selection
procedures of Dudewicz and Dalal \cite{Dudewicz1975} and Rinott
\cite{Rinott1978}, uses an initial sample to determine the size of a subsequent
sample so as to achieve a prescribed estimation accuracy. For a modern
perspective, see Jacobovic and Zuk \cite{Jacobovic2017}. Although motivated by
a fundamentally different estimation problem, the statistical role of the first
stage is conceptually similar. Here, the first-stage observations are not used
to estimate the service-time distribution directly, but rather to determine how
many observations may be safely retained from the second stage while preserving
the conditional independence structure required for the subsequent
decompounding step.

General references on statistical inference and inverse problems in queueing
systems include the surveys of Asanjarani, Nazarathy and Taylor
\cite{Asanjarani2021} and Baccelli, Kauffmann and Veitch
\cite{Baccelli2009}.

Viewed from this perspective, the present work combines ideas from queueing
inference, nonparametric decompounding, and classical two-stage statistical
procedures within a unified probabilistic framework.

\medskip

\noindent\textbf{Organization of the paper.} The remainder of the paper is organized as follows. In Section~\ref{sec: problem}, we introduce the M/G/1 workload model together with the associated nonparametric estimation problem. Section~\ref{sec: construction} develops the proposed two-stage screening procedure and explains how it reduces inference from the dependent workload observations to a classical decompounding problem. The main theoretical results are presented in Section~\ref{sec: result}, while their proofs are deferred to Section~\ref{sec: the proof}. Finally, Section~\ref{sec: discussion} concludes with a discussion of the proposed methodology, its limitations, and several directions for future research.
\medskip

\noindent\textbf{Notation.} For $a,b\in\mathbb{R}$, we write \[ a\vee b\equiv\max(a,b), \] and define \[ a^-\equiv\min(a,0), \] which differs from the more common convention $a^-=(-a)\vee0$. We denote by $\mathcal{D}([0,\infty))$ the Skorokhod space of c\`adl\`ag functions on $[0,\infty)$, and write $\mathrm{CP}_{\lambda,B}$ for the compound Poisson distribution with arrival rate $\lambda$ and jump-size distribution $B(\cdot)$. For any random object $\mathcal A$, let $\sigma(\mathcal A)$ denote the $\sigma$-field generated by $\mathcal A$. Moreover, for any nonnegative random variable $V$, we denote its Laplace--Stieltjes transform by \[ \widehat V(z)\equiv \mathbb Ee^{-zV}, \qquad z\in\mathbb C_{>0}, \] where \[ \mathbb C_{>0}\equiv \{z\in\mathbb C:\Re(z)>0\}. \] Unless stated otherwise, all random variables and stochastic processes are defined on a common probability space $(\Omega,\mathcal F,\mathbb P)$. Whenever convenient, we write $\mathbb P_x$ and $\mathbb E_x$ to indicate that the underlying stochastic process is initiated from state $x$. Throughout the paper, we use the standard asymptotic notation $\sim$, $\mathcal O$, and $\Theta$. Thus, for positive sequences $(a_n)$ and $(b_n)$, \[ a_n\sim b_n \quad\Longleftrightarrow\quad \lim_{n\to\infty}\frac{a_n}{b_n}=1, \] \[ a_n=\mathcal O(b_n) \quad\Longleftrightarrow\quad \sup_{n\ge1}\frac{a_n}{b_n}<\infty, \] and \[ a_n=\Theta(b_n) \quad\Longleftrightarrow\quad a_n=\mathcal O(b_n) \ \text{and}\ b_n=\mathcal O(a_n). \] Finally, for each $n\ge1$, we define \begin{equation}\label{eq:m_n} m_n\equiv \left\lfloor \frac{n-1}{2} \right\rfloor, \end{equation} and \begin{equation}\label{eq:m_bar} \overline m_n\equiv \left\lfloor \frac{m_n}{2} \right\rfloor. \end{equation}
These sequences are introduced here for convenience and will appear repeatedly throughout the remainder of the paper.
\section{Problem description}\label{sec: problem}

Let $J\equiv(J_t)_{t\geq0}$ be a compound Poisson process with arrival rate $\lambda>0$ and jump-size distribution function $B(\cdot)$ satisfying $B(0)=0$ such that $J_0=0$, $\mathbb{P}$-a.s. For a prescribed $x\geq0$, the workload process $W\equiv(W_t)_{t\geq0}$ under consideration is obtained as the Skorokhod reflection (see, \cite{Kella2006}) at the origin of the process $Z\equiv(Z_t)_{t\geq0}$,
\begin{equation}\label{eq: Z}
Z_t\equiv Z_t\left(J,x\right)\equiv x +J_t-t,\qquad t\geq0,
\end{equation}
that is,
\begin{equation}\label{eq: W}
W_t\equiv W_t\left[Z(J,x)\right]\equiv Z_t-\inf_{0\leq s\leq t} Z_s^-,\qquad t\geq0.    
\end{equation}
Unlike most existing work, no stability assumption is imposed.
In particular, we do not require the classical condition
\[
\lambda \int_0^\infty x{\rm d}B(x)<1.
\]

The workload process is not observed continuously. Instead, the data is
\begin{equation*}
\mathcal{D}_n\equiv\left\{
W_k
\,;\,
0\leq k\leq n
\right\},    
\end{equation*}
for some $n\geq1$. Apart from the observations, the statistician is assumed to know the
sampling scheme and the probabilistic structure of the model.
The only unknown quantities are $\lambda$ and $B(\cdot)$. Given a fixed point $w>0$, the objective is to estimate the value $B(w)$ with an estimator which is solely based on $\mathcal{D}_n$.

The difficulty of this problem stems from the fact that the information contained in the observations is highly indirect. Between two consecutive observation epochs, the statistician cannot determine
\begin{itemize}
\item the number of arrivals,

\item the corresponding jump sizes,

\item the amount of idle time.
\end{itemize}

Indeed, the path of the workload process between observation epochs remains completely unobserved. Consequently, the workload values observed at times $k-1$ and $k$ generally do not reveal whether the system became empty during the interval $(k-1,k)$.

The following examples illustrate the severity of the information loss. Observe that even if $W_{k-1}=W_k=0$,
it is still possible that several arrivals occurred between the two observation epochs and that all associated workload was subsequently processed before time $k$. Conversely, even when
\[
W_{k-1}>0
\qquad\text{and}\qquad
W_k>0,
\]
the system may nevertheless have emptied one or more times during the interval $(k-1,k)$. Hence, neither the arrival process nor the busy and idle periods can be reconstructed directly from the observations.

The remainder of the paper shows that this apparent loss of information
is not fundamental. By identifying observation intervals on which the Skorokhod reflection is
inactive, we recover an embedded conditionally independent compound
Poisson sample, reducing the original inference problem to a classical
decompounding problem. This reduction ultimately leads to an estimator whose $L^1$-risk is
parametric up to a logarithmic factor.

\section{Estimator construction}\label{sec: construction}

The estimator proposed in this paper is obtained by combining a screening step with an appropriate transformation of the observed data. The purpose of the transformation is to extract an i.i.d. sample whose Laplace--Stieltjes transform (LST) is linked explicitly to the LST of the unknown jump-size distribution $B(\cdot)$. This connection can then be exploited to construct an estimator for $B$.

Our approach relies on the general estimation methodology developed by Den Boer and Mandjes \cite{Den Boer2017}. We therefore begin by recalling the main ingredients of their construction. Subsequently, we demonstrate how the workload observations introduced in Section~\ref{sec: problem} can be processed so that the resulting transformed sample falls within their framework, thereby yielding our estimator. 

\subsection{Decompounding based on LST relations}\label{sec: Den Boer}

The construction of our estimator relies on a simple relationship between the LST's of a compound Poisson random variable and its jump-size distribution.

Let $X$ be a compound Poisson random variable with arrival rate $\lambda$ and jump-size distribution function $F^Y$, where $F^Y(0)=0$. Its LST is given by
\[
\widehat X(z)
=
\exp\!\left\{
-\lambda\left[1-\widehat Y(z)\right]
\right\},
\qquad
z\in\mathbb{C}_{>0},
\]
where $\widehat Y$ denotes the LST of the jump-size distribution.

Taking the principal branch of the complex logarithm yields
\[
\widehat Y(z)
=
1+\frac{1}{\lambda}
\operatorname{Log}\!\left[\widehat X(z)\right],
\qquad
z\in\mathbb{C}_{>0}.
\]
This identity motivates defining the mapping
\begin{equation}\label{eq: Psi identity}
\Psi(f)
\equiv
1+\frac{1}{\lambda
}\operatorname{Log}(f),
\end{equation}
so that the above identity may be written compactly as
\[
\widehat Y=\Psi(\widehat X).
\]

To recover the distribution function itself, define
\[
\overline{F}^{Y}(z)
=
\frac{\widehat Y(z)}{z},
\qquad
z\in\mathbb{C}_{>0}.
\]
Since
\[
\overline{F}^{Y}(z)
=
\int_0^\infty e^{-zt}F^{Y}(t)\,dt,
\qquad
z\in\mathbb{C}_{>0},
\]
the distribution function $F^{Y}$ is obtained from $\overline{F}^{Y}$ by Laplace inversion.

Suppose now that $X_1,\ldots,X_n$ are independent copies of $X$. Replacing $\widehat X$ by its empirical counterpart,
\[
\widehat X_n(z)
=
\frac1n
\sum_{j=1}^{n}
e^{-zX_j},
\qquad
z\in\mathbb{C}_{>0},
\]
leads naturally to the plug-in estimator
\[
\overline{F}^{Y}_n(z)
=
\frac1z
\Psi\!\left[\widehat X_n(z)\right].
\]

Finally, motivated by Bromwich's inversion formula, and adopting the convention that $\mathrm{undefined}\cdot0\equiv0$, we define
\begin{equation}
\label{eq: FY estimator}
F_n^Y(x)
\equiv
\mathbf 1_{E_n}
\frac1{2\pi}
\int_{-\sqrt n}^{\sqrt n}
e^{(c+iy)x}
\overline F_n^Y(c+iy)\,dy,
\qquad
x>0,
\end{equation}
where $c>0$ is an arbitrary constant and
\[
E_n
=
\left\{
\frac1n
\sum_{j=1}^{n}
\mathbf1_{\{X_j=0\}}
\in(0,1)
\right\}.
\]
The event $E_n$ guarantees that the empirical transform remains in the domain of the principal logarithm.

The following theorem, due to Den Boer and Mandjes \cite[Theorem~3]{Den Boer2017}, provides a non-asymptotic performance guarantee for the above estimator. 

\begin{theorem}[Den Boer and Mandjes (2017)]\label{thm: Den Boer}
Assume that $F^Y$ is continuously differentiable on $[0,\infty)$, twice differentiable at $w>0$, and satisfies
\[
\int_0^\infty y^2\,{\rm d}F^Y(y)<\infty.
\]
Then, there exists a constant $C>0$ such that
\[
\mathbb E\!\left|F_n^Y(w)-F^Y(w)\right|
\le
C\,\frac{\log(n+1)}{\sqrt n},
\]
for every $n\ge1$.
\end{theorem}

\subsection{Data transformation and screening}
We now introduce a transformation of the workload observations that identifies those sampling intervals during which the Skorokhod reflection remains inactive. On such intervals, the observed workload increment coincides exactly with the increment of the underlying compound Poisson input process. This observation provides the crucial link between the dependent workload observations and the decompounding methodology recalled in Section~\ref{sec: Den Boer}. To formalize this idea, for each $k\ge0$ define \[ I_k\equiv\mathbf1_{\{W_k>1\}}, \qquad X_k\equiv W_{k+1}-W_k+1. \] Since the sampling interval has unit length, the condition $W_k>1$ guarantees that the workload cannot reach the reflecting boundary during the interval $[k,k+1]$. Consequently, the Skorokhod reflection remains inactive throughout this interval, and \[ X_k=J_{k+1}-J_k. \] Hence, conditional on the event $\{I_k=1\}$, the random variable $X_k$ has the compound Poisson distribution $\mathrm{CP}_{\lambda,B}$. The estimator is constructed from a screened sample of such observations. To this end, define \[ K_n \equiv \sum_{k=0}^{m_n} I_k\mathbf1_{\{X_k>2\}}, \qquad L_n \equiv \sum_{k=m_n+1}^{n} I_k, \] where $m_n$ is defined in \eqref{eq:m_n}. The random variable $K_n$ counts the number of screened observations satisfying $X_k>2$ in the first stage of the sample, whereas $L_n$ counts the number of screened observations available in the second stage. The first-stage observations are used solely to determine the effective sample size, while the second-stage observations are used exclusively for estimation. This separation is a key ingredient of the construction, as it ensures that the random sample size employed by the estimator is selected independently of the observations to which the estimation procedure is ultimately applied. Next, let $\tau_i$ denote the $i$th index after time $m_n$ satisfying the screening condition, namely \[ \tau_1 \equiv \inf\{k\ge m_n+1:\,I_k=1\}, \] and, for $i\ge2$, \[ \tau_i \equiv \inf\{k>\tau_{i-1}:\,I_k=1\}. \] 
From the analysis presented in Section~\ref{sec: the proof}, it will follow that $\tau_i<\infty$ almost surely for every $i\geq 1$. Consequently, the sequence
\[
\tau_1,\tau_2,\ldots
\]
is well defined and enumerates the indices of the second-stage observations that satisfy the screening criterion. The corresponding random variables
\[
X_{\tau_1},X_{\tau_2},\ldots
\]
constitute the screened sample used for estimation. Using these observations, we get the empirical Laplace--Stieltjes transform \begin{equation}\label{eq: X_n} \widehat X_n(z) = \frac1n \sum_{k=1}^{n} e^{-zX_{\tau_k}}, \qquad z\in\mathbb C_{>0}, \end{equation} together with the admissibility event \begin{equation}\label{eq: E_n} E_n = \left\{ \frac1n \sum_{k=1}^{n} \mathbf1_{\{X_{\tau_k}=0\}} \in(0,1) \right\}. \end{equation} We are now in a position to define the proposed estimator: \begin{align} B_n(w) &\equiv \mathbf1_{\{K_n<L_n\}} \mathbf1_{\{K_n\ge4\}} \,\overline B_n(w) \label{eq: estimator}\\ &\equiv \mathbf1_{\{K_n<L_n\}} \mathbf1_{\{K_n\ge4\}} \mathbf1_{E_{K_n}} \frac1{2\pi} \int_{-\sqrt{K_n}}^{\sqrt{K_n}} e^{(c+iy)w} \, \overline F_{K_n}^{Y}(c+iy) \,dy. \nonumber \end{align} The construction of $B_n$ proceeds as follows. We first check whether \[ K_n<L_n \qquad\text{and}\qquad K_n\ge4. \] If this event occurs, we apply the inversion procedure of Section~\ref{sec: Den Boer} to the first $K_n$ screened observations, namely \begin{equation}\label{eq: subsequence} X_{\tau_1},X_{\tau_2},\ldots,X_{\tau_{K_n}}, \end{equation} using the empirical transform \eqref{eq: X_n} together with the admissibility event \eqref{eq: E_n}. The resulting estimator is denoted by $\overline B_n$. As will be established in the next section, conditional on the first-stage history \begin{equation}\label{eq: history} \mathcal H_{m_n+1}\equiv \sigma(\mathcal D_{m_n+1}), \end{equation} the random variables in \eqref{eq: subsequence} are independent and identically distributed with common distribution $\mathrm{CP}_{\lambda,B}$. Thus, conditional on $\mathcal H_{m_n+1}$, the distribution of the screened data is consistent with the setup discussed in Section~\ref{sec: Den Boer}. Consequently, the choice of the decompounding estimator $\overline B_n$ by Den Boer and Mandjes becomes very intuitive. If the event \[ \{K_n<L_n,\;K_n\ge4\} \] does not occur, we simply set $B_n\equiv0$. Since this decision depends only on the observed data $\mathcal D_n$, the estimator is well defined and fully data-driven.
\begin{remark} \normalfont A natural alternative estimator construction would be to bypass the first screening stage and apply the estimator of Section~\ref{sec: Den Boer} directly to the entire screened sample \[ \{X_k;\, I_k=1,\;0\leq k\leq n-1\}. \] At first sight, this construction appears more attractive since it retains all available observations. However, the resulting sample does not satisfy the probabilistic assumptions required by the decompounding procedure. More precisely, conditional on its random size, the retained observations are generally not i.i.d. with the common compound Poisson distribution. The difficulty stems from the dependence between the screening indicators and the workload trajectory. Conditioning on the event that exactly $m$ observations are retained, the screened sample no longer has the joint distribution of an i.i.d. compound Poisson sample. The first screening stage is therefore essential, as it decouples the random sample size from the screened observations, thereby ensuring the sample properties required for the decompounding analysis.
\end{remark}

\section{Main result}\label{sec: result} We now state the main result of the paper. It establishes that the proposed estimator achieves a nearly parametric $L^1$-risk under remarkably mild assumptions. In particular, neither stationarity nor stability of the queue is required, and the arrival rate may remain unknown. The proof is deferred to Section~\ref{sec: the proof}. \begin{theorem}\label{thm: main} Assume that $B(0)=0$, that $B(\cdot)$ is continuously differentiable on $[0,\infty)$, twice differentiable at a prescribed point $w>0$, and satisfies \[ \int_0^\infty y^2\,{\rm d}B(y)<\infty. \] Then, \[ \mathbb E\!\left|B_n(w)-B(w)\right| = \mathcal O\!\left( \frac{\log n}{\sqrt n} \right), \qquad n\to\infty. \] \end{theorem} \begin{remark} \normalfont The assumptions of Theorem~\ref{thm: main} are remarkably mild. Most notably, neither stationarity nor stability of the queue is required. Furthermore, the only global regularity assumptions imposed on the service-time distribution are continuous differentiability together with the finite second-moment condition \[ \int_0^\infty y^2\,{\rm d}B(y)<\infty. \] The second differentiability assumption is purely local, being required only at the estimation point $w$. Finally, the condition $B(0)=0$ simply reflects the fact that service times are strictly positive, and is standard in the analysis of M/G/1 queues. \end{remark} 

\begin{remark} \normalfont The assumption of unit sampling intervals is made solely for notational convenience. Indeed, if the workload process is observed at times \[ 0,\delta,2\delta,\ldots,n\delta, \] for some $\delta>0$, then a simple linear rescaling of time reduces the model to the unit-time setting considered throughout the paper. In particular, the screening condition becomes \[ I_k=\mathbf1_{\{W_{k\delta}>\delta\}}, \] which again guarantees that the Skorokhod reflection remains inactive over the selected sampling intervals. Consequently, all arguments presented in the paper remain valid after this straightforward modification, and the conclusion of Theorem~\ref{thm: main} continues to hold without change. \end{remark}

\section{Proof of Theorem~\ref{thm: main}}\label{sec: the proof}
This section is devoted to the proof of Theorem~\ref{thm: main}. After the outline of the proof given below, we establish a sequence of auxiliary results describing the probabilistic properties of the proposed screening procedure. These results culminate in the conditional independence property required for the decompounding methodology of Theorem~\ref{thm: Den Boer}. The proof of the main theorem is then obtained by combining these ingredients with suitable probabilistic estimates for the screening mechanism.

\subsection{Outline of the proof} The proof of Theorem~\ref{thm: main} is based on a probabilistic reduction of the original inference problem. Rather than analyzing the dependent workload observations directly, we first show that the proposed two-stage screening procedure extracts, conditionally on the first-stage history, a sample of independent and identically distributed compound Poisson random variables. This reduction transforms the original estimation problem into a classical decompounding problem, allowing us to exploit the methodology recalled in Section~\ref{sec: Den Boer}. The main technical challenge is to justify this reduction rigorously. To this end, we first establish several structural properties of the screening procedure. In particular, we prove that, conditional on the first-stage history, the screened observations retained for estimation are independent and identically distributed with common law $\mathrm{CP}_{\lambda,B}$. Subsequently, we derive quantitative bounds showing that the probability of an unsuccessful screening event decreases sufficiently rapidly with the sample size. Once these ingredients have been established, the proof of Theorem~\ref{thm: main} follows by combining the conditional decompounding result of Den Boer and Mandjes with the probabilistic estimates obtained for the screening procedure. The logarithmic factor in the convergence rate is therefore inherited entirely from the decompounding estimator, while the screening step itself does not introduce any additional loss in the order of convergence.

\subsection{Coupling}
The purpose of this part is to construct a coupling between the workload process and suitable sequences of independent Bernoulli random variables. This coupling provides stochastic upper and lower bounds for the random sample sizes $K_n$ and $L_n$, which will play a crucial role in the asymptotic analysis of the estimator.

Throughout the remainder of the paper, we shall work with the following coupling construction of the workload process. Let
$(\Omega,\mathcal F,\mathbb P)$
be a probability space supporting an i.i.d. sequence
$J^1,J^2,\ldots$
of stochastic processes distributed as $J$.
Since the workload process is a strong Markov process (see, Corollary~2.8 of Chapter~IX in \cite{Asmussen2003}), the workload process admits the recursive representation
\begin{equation}\label{eq: recursion}
W_k
=\begin{cases}
    x & \text{if}\quad k=0 \\
    W_1\!\left[Z(J^k,W_{k-1})\right]  & \text{if}\quad k\geq1
\end{cases}\,.
\end{equation}

\begin{lemma}\label{lem: coupling}
There exist  $0<q<p<1$ and three random variables $A_n,K_n^*$ and $V_n$ such that:
\begin{enumerate}
    \item[(i)] $A_n\leq L_n$ almost surely and 
    \begin{equation*}
        A_n\sim\text{\normalfont Bin}(n-m_n,q).
    \end{equation*}

    \item[(ii)] $K_n\leq V_n$ almost surely such that 
    \begin{equation*}
    V_n\sim\text{\normalfont Bin}(m_n+1,p).
    \end{equation*}
    \item[(iii)] $K_n^*\leq K_n$ almost surely such that 
   \begin{equation*}
    K_n^*\sim\text{\normalfont Bin}(\overline m_n+1,p^2).
    \end{equation*}
    
     \item[(iv)] $A_n$ and $(K_n^*,V_n)$ are independent.  \end{enumerate}

\end{lemma}
\begin{proof}
Recall the processes
$t\mapsto Z_t(x,J)$
and
$t\mapsto W_t[Z(J,x)]$
defined in
\eqref{eq: Z}
and
\eqref{eq: W}.
Since
\[
Z_t(x,J)=x+J_t-t,
\]
the mapping
$x\mapsto Z_t(x,J)$
is increasing for every fixed $t\ge0$.
Moreover, the Skorokhod reflection map is monotone with respect to the pointwise order on
$\mathcal D([0,\infty))$.
Consequently,
\[
x\longmapsto W_t[Z(J,x)]
\]
is also increasing.

We now construct an auxiliary workload process by restarting the queue from zero after every observation epoch. Define
\[
W_k^*
=
W_1[Z(J^k,0)],
\qquad
k\ge1,
\]
with $W_0^*=0$.
Since $W_{k-1}\ge0$, monotonicity yields
\[
W_k^*\le W_k,
\qquad
k\ge1.
\]
Hence,
\[
I_k^*
\equiv
\mathbf1_{\{W_k^*>1\}}
\le
I_k,
\qquad
k\ge1,
\]
almost surely.

Observe that $W_k^*$ depends only on $J^k$.
Therefore,
$(I_k^*)_{k\ge1}$
is an i.i.d. sequence.
Furthermore,
\begin{align*}
q
&\equiv
\mathbb P(I_k^*=1)
\\
&=
\mathbb P\!\left(
W_1[Z(J^1,0)]>1
\right)
\\
&>
\mathbb P\!\left(
Z_1(J^1,0)>1
\right)
\\
&=
\mathbb P(J_1^1>2)
\equiv
p.
\end{align*}
In particular, the strict inequality holds since when starting at zero, the reflection remains active until the first jump.
Thus,
\[
I_k^*\sim\mathrm{Ber}(q).
\]

Define
\[
A_n
=
\sum_{k=m_n+1}^{n}I_k^*.
\]
Since $I_k^*\le I_k$,
\[
A_n\le L_n
\]
almost surely, while
\[
A_n\sim\mathrm{Bin}(n-m_n,q),
\]
proving (i).

To establish (ii), observe that
\[
I_k\mathbf1_{\{X_k>2\}}
=
I_k\mathbf1_{\{J^{k+1}_1>2\}}
\le
\mathbf1_{\{J^{k+1}_1>2\}}.
\]
Hence,
\[
V_n\equiv\sum_{k=0}^{m_n}
\mathbf1_{\{J^{k+1}_1>2\}}
\]
satisfies
\[
K_n\le V_n
\]
almost surely.
Since the indicators are i.i.d. Ber$(p)$,
\[
V_n\sim\mathrm{Bin}(m_n+1,p),
\]
which proves (ii).

For (iii), note that
\[
I_k
\mathbf1_{\{X_k>2\}}
=
I_k
\mathbf1_{\{J^{k+1}_1>2\}}
\ge
I_k^*
\mathbf1_{\{J^{k+1}_1>2\}}.
\]
Moreover,
\[
I_k^*
\geq
\mathbf1_{\{J^{k+1}_1>2\}}
,
\]
since
$J^{k+1}_1>2$
implies
$W_k^*>1$.
Therefore,
\[
I_k
\mathbf1_{\{X_k>2\}}
\ge
\mathbf1_{\{J^{k+1}_1>2\}}
\mathbf1_{\{J^{k+1}_1>2\}}.
\]
Define
\[
K_n^*\equiv
\sum_{k=0}^{\overline m_n}
\mathbf1_{\{J^{2k+1}_1>2\}}
\mathbf1_{\{J^{2k+2}_1>2\}}.
\]
The summands involve disjoint pairs of independent processes and are therefore i.i.d. Ber$(p^2)$.
Furthermore,
\[
K_n^*
\le
K_n
\]
almost surely.
Hence,
\[
K_n^*
\sim
\mathrm{Bin}(\overline m_n+1,p^2),
\]
establishing (iii).

Finally,
$A_n$
depends only on
\[
J^{m_n+1},
J^{m_n+2},
\ldots,
\]
whereas
$(K_n^*,V_n)$
depends only on
\[
J^1,\ldots,J^{m_n}.
\]
Since these two collections of driving processes are independent, so are
$A_n$
and
$(K_n^*,V_n)$,
which proves (iv).
\end{proof}

\subsection{An embedded conditionally i.i.d. sample}
Although the workload observations are inherently dependent, the screening procedure introduced in the previous subsection reveals an embedded sequence with a remarkably simple probabilistic structure. Indeed, by retaining only those sampling intervals during which reflection is inactive, the corresponding workload increments become independent realizations of the underlying compound Poisson input process. The next lemma makes this observation precise and forms the cornerstone of the proposed estimation procedure.
\begin{lemma}\label{lemma: embedded iid}
Conditionally on $\mathcal H_{m_n+1}$, the sequence
\[
(X_{\tau_k})_{k\ge1}
\]
is i.i.d. with common distribution $\mathrm{CP}_{\lambda,B}$.
\end{lemma}

\begin{proof}
Denote 
\[
\mathcal H_n\equiv
\sigma(\mathcal{D}_n),
\qquad n\ge0.
\]

Suppose that the workload at the beginning of a sampling interval satisfies
${W_0>1}$.
Since the server removes one unit of work over an interval of length one, the workload remains strictly positive throughout $(0,1)$.
Consequently, the reflection term is inactive over this interval and
\begin{equation}\label{eq: CP proof}
W_1-W_0+1
=
J_1
\sim
\mathrm{CP}_{\lambda,B}.
\end{equation}

Now recall from \eqref{eq: recursion} that
$(W_k)_{k\ge0}$
is a time-homogeneous Markov chain.
Since each $\tau_k$ is a stopping time with respect to the filtration
$(\mathcal H_n)_{n\ge0}$,
the strong Markov property yields, for every Borel set
$\mathcal A\subset[0,\infty)$,
\begin{align}
\mathbb P_x
\!\left(
X_{\tau_k}\in \mathcal A
\,\middle|\,
\mathcal H_{\tau_k}
\right)
&=
\mathbb P_{W_{\tau_k}}
\!\left(
W_1-W_0+1\in \mathcal A
\right)
\nonumber\\
&=
\mathrm{CP}_{\lambda,B}(\mathcal A),
\label{eq: cpp-markov}
\end{align}
where the last equality follows from \eqref{eq: CP proof}, since
$W_{\tau_k}>1$
by definition of $\tau_k$.

It remains to prove conditional independence.
Let
$\mathcal A_1,\ldots,\mathcal A_r$
be Borel subsets of $[0,\infty)$.
Using \eqref{eq: cpp-markov}, the tower property argument, we obtain
\begin{align*}
&
\mathbb P_x
\left(
X_{\tau_i}\in\mathcal  A_i,\;
1\le i\le r
\,\middle|\,
\mathcal H_{m_n+1}
\right)
\\
&=
\mathbb E_x
\left[
\prod_{i=1}^{r-1}
\mathbf1_{\{X_{\tau_i}\in \mathcal A_i\}}
\,
\mathbb P_x
\left(
X_{\tau_r}\in \mathcal A_r
\,\middle|\,
\mathcal H_{\tau_r}
\right)
\middle|
\mathcal H_{m_n+1}
\right]
\\
&=
\mathrm{CP}_{\lambda,B}(\mathcal A_r)\,
\mathbb P_x
\left(
X_{\tau_i}\in \mathcal A_i,\;
1\le i\le r-1
\,\middle|\,
\mathcal H_{m_n+1}
\right).
\end{align*}
Iterating this identity yields
\[
\mathbb P_x
\left(
X_{\tau_i}\in \mathcal A_i,\;
1\le i\le r
\,\middle|\,
\mathcal H_{m_n+1}
\right)
=
\prod_{i=1}^{r}
\mathrm{CP}_{\lambda,B}(\mathcal A_i),
\]
which proves that
$(X_{\tau_k})_{k\ge1}$
is conditionally i.i.d. with common distribution
$\mathrm{CP}_{\lambda,B}$.
\end{proof}

\subsection{Risk bound}
 Since both $B(w)$ is a value in $[0,1]$, the $L^1$-risk of the estimator admits the decomposition
\begin{equation}\label{eq: risk decomposition}
\mathbb E|B_n(w)-B(w)|
\le
\mathbb P\{L_n\le K_n\}
+
\mathbb P\{K_n<4\}
+
\mathbb E|\overline B_n(w)-B(w)|\textbf{1}_{\{K_n\geq4\}}.
\end{equation}
The remainder of the proof consists of bounding each of the three terms on the right-hand side of \eqref{eq: risk decomposition}. 

\subsubsection{Bounding $\mathbb P\{L_n\le K_n\}$ and
$\mathbb P\{K_n<4\}$}

We first control the two probability terms in \eqref{eq: risk decomposition}. The bounds rely on the coupling established in Lemma~\ref{lem: coupling}, together with Hoeffding's inequality.  

Namely, due to Lemma~\ref{lem: coupling},
\[
A_n\sim\mathrm{Bin}(n-m_n,q),
\qquad
V_n\sim\mathrm{Bin}(m_n+1,p).
\]
Thus, we obtain
\[
\mathbb E(A_n-V_n)
=
(n-m_n)q-(m_n+1)p.
\]
Moreover,
\[
n-m_n
=
\left\lceil\frac{n+1}{2}\right\rceil,
\]
so that
\[
\mathbb E(A_n-V_n)
=
\frac{q-p}{2}\,n+\mathcal O(1).
\]
Consequently, for sufficiently large $n$,
\[
\mathbb E(A_n-V_n)
\ge
\frac{q-p}{4}\,n,
\]
Now, since $A_n$ and $V_n$ are independent, we write
\[
A_n-V_n
=
\sum_{i=1}^{n-m_n}\xi_i
-
\sum_{j=1}^{m_n+1}\eta_j,
\]
where
\[
\xi_i\stackrel{\mathrm{i.i.d.}}{\sim}\mathrm{Ber}(q),
\qquad
\eta_j\stackrel{\mathrm{i.i.d.}}{\sim}\mathrm{Ber}(p),
\]
and all variables are mutually independent.

Since every summand belongs to $[-1,1]$, Hoeffding's inequality and Lemma~\ref{lem: coupling} jointly provide
\begin{align*}
\mathbb P\{L_n\le K_n\}
&
\le
\mathbb P\{A_n\le V_n\}
\\
&
=
\mathbb P\!\left\{
A_n-V_n-\mathbb E(A_n-V_n)
\le
-\mathbb E(A_n-V_n)
\right\}
\\
&
\le
\exp\!\left[
-\frac{\mathbb E^2(A_n-V_n)}
{2(n-m_n+m_n+1)}
\right]
\\
&
\le
\exp\!\left[
-\frac{(q-p)^2}{32}\,n
\right],
\end{align*}
for all sufficiently large $n$.

Similarly, Lemma~\ref{lem: coupling} gives
\[
\mathbb P\{K_n<4\}
\le
\mathbb P\{K_n^*<4\},
\]
where
\[
K_n^*
\sim
\mathrm{Bin}(\overline m_n+1,p^2).
\]
Since
\[
\overline m_n+1
=
\left\lfloor\frac{m_n}{2}\right\rfloor
=
\frac{n}{4}+\mathcal O(1),
\]
we have
\[
\mathbb EK_n^*
=
(\overline m_n+1)p^2
=
\frac{p^2}{4}\,n+\mathcal O(1).
\]
Hence for any sufficiently large $n$ we have
\[
\mathbb EK_n^*
\ge
\frac{p^2}{5}\,n.
\]

Applying Hoeffding's inequality once more,
\begin{align*}
\mathbb P\{K_n<4\}
&
\le
\mathbb P\{K_n^*<4\}
\\
&
=
\mathbb P\!\left\{
K_n^*-\mathbb EK_n^*
<
4-\mathbb EK_n^*
\right\}
\\
&
\le
\exp\!\left(
-\frac{\mathbb E^2(K_n^*-4)}
{2(\overline m_n+1)}
\right)
\\
&
\le
\exp\!\left(
-\frac{2p^4}{25}n\,
\right),
\end{align*}
for all sufficiently large $n$.

\subsubsection{Bounding $\mathbb E
|\overline B_n(w)-B(w)|
\mathbf1_{\{K_n\ge4\}}
$}

It remains to bound the third term in the risk decomposition
\eqref{eq: risk decomposition}.
For notational convenience, define
\[
g(x)\equiv
\frac{\log(x+1)}{\sqrt{x}},
\qquad
x>0.
\]
A straightforward calculation shows that
$g$
is decreasing on
$[4,\infty)$.

Recall that
$\overline B_n$
is precisely the estimator introduced by
Den Boer and Mandjes~\cite{Den Boer2017},
applied to the screened sample
\eqref{eq: subsequence}.
Furthermore,
Lemma~\ref{lemma: embedded iid}
shows that, conditionally on
$\mathcal H_{m_n+1}$,
the random variables
\[
X_{\tau_1},X_{\tau_2},
\ldots
\]
are i.i.d. with common distribution
$\mathrm{CP}_{\lambda,B}$. In addition, $K_n$ is measurable with respect to $\mathcal{H}_{m_n+1}$. Therefore, conditionally on
$\mathcal H_{m_n+1}$,
the observations in the screened sample \eqref{eq: subsequence} are i.i.d. with common distribution
$\mathrm{CP}_{\lambda,B}$. 
Since the assumptions imposed on
$B$
imply those of
Theorem~\ref{thm: Den Boer},
we may apply the latter conditionally on
$\mathcal H_{m_n+1}$.

Consequently,
\begin{align}
&
\mathbb E
|\overline B_n(w)-B(w)|
\mathbf1_{\{K_n\ge4\}}
\nonumber
\\
&
=
\mathbb E\mathbf1_{\{K_n\ge4\}}
\,
\mathbb E\left[
|\overline B_n(w)-B(w)|
\,\middle|\,
\mathcal H_{m_n+1}
\right]
\nonumber
\\
&
\le
C
\,
\mathbb E
\mathbf1_{\{K_n\ge4\}}
g(K_n).
\label{eq: second term 1}
\end{align}

Since
$K_n^*\le K_n$
almost surely and
$g$
is nonnegative on $[0,\infty)$ and decreasing on
$[4,\infty)$,
we have
\[
\mathbf1_{\{K_n\ge4\}}
g(K_n)
\le g(K_n\vee4)\le
g(K_n^*\vee4).
\]
Therefore,
\[
\mathbb E
|\overline B_n(w)-B(w)|
\mathbf1_{\{K_n\ge4\}}
\le
C
\,
\mathbb Eg(Z_n),
\]
where
\[
Z_n\equiv K_n^*\vee4.
\]

Hence, it only remains to estimate
$\mathbb Eg(Z_n)$.
The next lemma provides the required moment estimates.

\begin{lemma}\label{lemma: central moments asymptotics}
For every integer
$r\ge1$,
define
\[
\mu_n^r=
\begin{cases}
\mathbb EZ_n,
&
r=1,
\\
\mathbb E|Z_n-\mu_n^1|^r,
&
r\ge2.
\end{cases}
\]
Then,
\[
0
\le
\mu_n^1
-
(\overline m_n+1)p^2
\le
4,
\]
and
\[
\mu_n^1
\sim
\frac{p^2}{4}n,
\qquad
\mu_n^r
=
\mathcal O(n^{r/2}),
\qquad
r\ge2.
\]
\end{lemma}

\begin{proof}
Recall that
\[
Z_n
=
K_n^*\vee4,
\]
where
\[
K_n^*
\sim
\mathrm{Bin}(\overline m_n+1,p^2).
\]
Since
\[
K_n^*
\le
Z_n
\le
K_n^*+4,
\]
we immediately obtain
\[
(\overline m_n+1)p^2
\le
\mu_n^1
\le
(\overline m_n+1)p^2+4.
\]
Since
\[
\overline m_n+1
=
\frac n4+\mathcal O(1),
\]
it follows that
\[
\mu_n^1
\sim
\frac{p^2}{4}n.
\]

Next, write
\[
K_n^*
=
\sum_{k=0}^{\overline m_n}\xi_k,
\]
where
$\xi_0,\xi_1,\ldots$
are i.i.d.
Bernoulli random variables with parameter
$p^2$.
Define
\[
S_n\equiv
\sum_{k=0}^{\overline m_n}
(\xi_k-p^2).
\]
Then
\[
K_n^*
=
S_n
+
(\overline m_n+1)p^2.
\]
Since
\[
\mu_n^1
\in
\left[
(\overline m_n+1)p^2,
(\overline m_n+1)p^2+4
\right],
\]
we obtain
\[
|Z_n-\mu_n^1|
\le
|S_n|+8.
\]
Hence, for every
$r\ge2$,
\[
|Z_n-\mu_n^1|^r
\le
2^r
\left(
|S_n|^r+8^r
\right),
\]
and therefore
\[
\mu_n^r
\le
2^r
\left(
\mathbb E|S_n|^r
+
8^r
\right).
\]
Finally,
since the variables
$\xi_k-p^2$
are centered,
independent,
and uniformly bounded,
Rosenthal's inequality yields
\[
\mathbb E|S_n|^r
=
\mathcal O(n^{r/2}),
\qquad
n\to\infty.
\]
Consequently,
\[
\mu_n^r
=
\mathcal O(n^{r/2}),
\]
which completes the proof.
\end{proof}
$\text{ }$\newline\newline
We now complete the proof of Theorem~\ref{thm: main} by estimating
$\mathbb Eg(Z_n)$ using second order delta method argument.
Since
$g\in C^2((0,\infty))$,
Taylor's theorem (see, \cite[Theorem 1]{Yang2026}) yields
\begin{equation}\label{eq:Taylor}
g(Z_n)
=
g(\mu_n^1)
+
g'(\mu_n^1)(Z_n-\mu_n^1)
+
R_n,
\end{equation}
where
\[
R_n
=
\frac12
g''(\eta_n)
(Z_n-\mu_n^1)^2,
\]
for some random variable
$\eta_n$
lying between
$Z_n$
and
$\mu_n^1$.

Since
$\mu_n^1=\mathbb EZ_n$,
\begin{equation}\label{eq:mean-term}
\mathbb E(Z_n-\mu_n^1)=0.
\end{equation}

It therefore remains to estimate
$\mathbb E|R_n|$.

Since
\[
g''(x)
=
\mathcal O
\!\left(
\frac{\log x}{x^{5/2}}
\right),
\qquad
x\to\infty,
\]
for any $x_0>0$ there is a constant
$c_1>0$
such that
\[
|g''(x)|
\le
c_1
\frac{\log x}{x^{5/2}},
\qquad
x\ge x_0.
\]

Choose
$\delta,\varepsilon>0$
such that
\[
\frac{p^2}{4}-\delta-\varepsilon>0.
\]
Since
\[
\mu_n^1
\sim
\frac{p^2}{4}n,
\]
there exists
$n_0$
such that
\[
\left(
\frac{p^2}{4}-\delta
\right)n
<
\mu_n^1
<
\left(
\frac{p^2}{4}+\delta
\right)n,
\qquad
n\ge n_0.
\]

Define
\[
C_n
=
\left\{
|\eta_n-\mu_n^1|
\le
\varepsilon n
\right\}.
\]

Thus, for any sufficiently large $n$, on the event
$C_n$,
we have
\[
\eta_n
\ge
\left(
\frac{p^2}{4}
-
\delta
-
\varepsilon
\right)n,
\]
and therefore
\[
|g''(\eta_n)|
\le
c_1
\frac{\log n}{n^{5/2}}.
\]
Consequently, for any sufficiently large $n$,
\[
|R_n|
\mathbf1_{C_n}
\le
c_1
\frac{\log n}{n^{5/2}}
(Z_n-\mu_n^1)^2,
\]
As a result, by taking expectations and applying
Lemma~\ref{lemma: central moments asymptotics}
gives
\begin{equation}\label{eq:R-A}
\mathbb E
|R_n|
\mathbf1_{C_n}
=
\mathcal O
\!\left(
\frac{\log n}{n^{3/2}}
\right).
\end{equation}

Next, observe that
$\eta_n$
lies between
$Z_n$
and
$\mu_n^1$.
Hence, for any $n$ we have
\[
C_n^c
\subseteq
\left\{
|Z_n-\mu_n^1|
>
\varepsilon n
\right\}.
\]
Since
\[
K_n^*
\le
Z_n
\le
K_n^*+4,
\]
we further obtain
\[
C_n^c
\subseteq
\left\{
|K_n^*-\mu_n^1|
>
\varepsilon n-4
\right\},
\]
for any sufficiently large $n$.
Moreover,
Lemma~\ref{lemma: central moments asymptotics}
shows that,
\[
|\mu_n^1-(\overline m_n+1)p^2|
\le4,
\]
which implies also
\[
C_n^c
\subseteq
\left\{
|K_n^*-(\overline m_n+1)p^2|
>
\varepsilon n-8
\right\}.
\]
For all sufficiently large
$n$,
\[
\varepsilon n-8
\ge
\frac{\varepsilon n}{2},
\]
and therefore
\[
C_n^c
\subseteq
\left\{
|K_n^*-(\overline m_n+1)p^2|
>
\frac{\varepsilon n}{2}
\right\}.
\]
Applying Hoeffding's inequality to the
binomial random variable $K_n^*$ yields
\begin{align*}
    \mathbb P
\!\left(
|K_n^*-(\overline m_n+1)p^2|
>
\frac{\varepsilon n}{2}
\right)
&\le
2
\exp
\!\left(
-
\frac{(\varepsilon n/2)^2}
{2(\overline m_n+1)}
\right)
\\[4pt]&=
\exp
\!\left(
-
\frac{\varepsilon^2n}
{2}
\right)
.
\end{align*}
Therefore, for any sufficiently large $n$, we derive an upper bound
\begin{equation}\label{eq:Hoeffding-An}
\mathbb P(C_n^c)
\le
2
\exp
\!\left(
-
\frac{\varepsilon^2n}
{2}
\right)
\end{equation}

On the other hand,
continuity of
$g''$
on
$[4,\infty)$
implies that
\[
|g''(x)|
\le
c_2,
\qquad
x\ge4,
\]
for some constant
$c_2>0$.
Hence, for any sufficiently large $n$, 
\[
|R_n|
\le
c_2
(Z_n-\mu_n^1)^2.
\]
As a result, by applying the Cauchy--Schwarz inequality, deduce that
\begin{align}
\mathbb E
|R_n|
\mathbf1_{C_n^c}
&\le
c_2^2
\sqrt{
\mathbb P(C_n^c)
\,
\mathbb E
(Z_n-\mu_n^1)^4
}
\nonumber\\
&=
\mathcal O
\!\left(
n
e^{-\frac{\varepsilon^2n}{2}}
\right),
\label{eq:R-Ac}
\end{align}
where the last estimate follows from
Lemma~\ref{lemma: central moments asymptotics}.

Combining
\eqref{eq:R-A}
and
\eqref{eq:R-Ac},
we conclude that
\[
\mathbb E|R_n|
=
\mathcal O
\!\left(
\frac{\log n}{n^{3/2}}
\right).
\]

Taking expectations in
\eqref{eq:Taylor},
using
\eqref{eq:mean-term},
we obtain
\[
\mathbb Eg(Z_n)
=
g(\mu_n^1)
+
\mathcal O
\!\left(
\frac{\log n}{n^{3/2}}
\right).
\]
Finally,
since
\[
\mu_n^1
\sim
\frac{p^2}{4}n,
\]
we have
\[
g(\mu_n^1)
=
\frac{\log(\mu_n^1+1)}
{\sqrt{\mu_n^1}}
=
\mathcal O
\!\left(
\frac{\log n}{\sqrt n}
\right).
\]
Therefore,
\[
\mathbb Eg(Z_n)
=
\mathcal O
\!\left(
\frac{\log n}{\sqrt n}
\right),
\]
which completes the proof of
Theorem~\ref{thm: main}. $\blacksquare$

\section{Discussion}\label{sec: discussion} This paper resolves the nonparametric estimation problem posed by Hansen and Pitts \cite{Hansen2006} for the M/G/1 workload model without assuming stationarity, stability, or knowledge of the arrival rate~$\lambda$. In doing so, it demonstrates that accurate nonparametric inference from workload observations is possible without relying on the classical assumptions that have traditionally guided statistical analysis of queueing systems. We hope that this work will contribute to a broader shift in perspective, encouraging the development of inference procedures that remain valid beyond the stationary and stable setting.


In addition, we believe that a major contribution of this paper is methodological rather than algorithmic. Instead of attacking the dependent workload observations directly, we introduced a two-stage screening procedure that performs a probabilistic reduction of the original inference problem. Conditionally on the first-stage history, the screened observations become independent compound Poisson random variables, thereby transforming the workload estimation problem into a classical decompounding problem. The resulting estimator is therefore obtained by combining this probabilistic reduction with an existing decompounding methodology. 

This separation between the probabilistic and statistical components of the construction has an important conceptual consequence. The screening and coupling arguments are essentially independent of the particular decompounding procedure employed in the second stage. Consequently, the framework developed here is inherently modular: any future improvement in the underlying decompounding methodology can, in principle, be incorporated into the present construction and immediately translated into an improved estimator for the workload model using the same proof guidelines. In particular, the logarithmic factor in Theorem~\ref{thm: main} originates entirely from the currently available decompounding estimator rather than from the screening procedure itself. 

The sample-splitting construction also gives rise to several interesting questions.
Its main role in the present framework is to decouple the random selection of the
effective sample size from the observations subsequently used for estimation,
thereby ensuring the conditional independence structure required by our analysis.
An important open question is whether such a separation is an intrinsic feature of
the problem or merely an artifact of the proof strategy employed here. More
specifically, it remains to be understood whether a genuinely one-stage estimator
can attain the same convergence rate, or perhaps even improve upon it, under the
same level of generality as the assumptions imposed in this paper. To the best of our knowledge, this is the first use of a sample-splitting construction in the context of statistical inference for queueing systems. More broadly, sample splitting has received considerable attention as a methodological tool in statistics. It is typically viewed as a flexible device for separating different components of an inference procedure, although such a separation may come at the cost of reduced statistical efficiency; see, for example, Cox \cite{Cox1975}, Moran \cite{Moran1973}, Jacobovic \cite{Jacobovic2022}, and Goeman and Solari \cite{Goeman2024}, as well as the references therein.

Another natural question concerns the optimal choice of the splitting proportion. Although we adopted the symmetric choice of $m_n$ for simplicity, other choices may improve finite-sample performance. Our analysis suggests that such refinements would affect mainly the constants rather than the convergence rate.

Finally, we believe that the probabilistic reduction developed here extends well beyond the specific estimation problem studied in this paper. It would be particularly interesting to investigate whether similar screening mechanisms can be constructed for broader classes of reflected stochastic processes, including Lévy-driven storage models, reflected diffusion processes, and queueing networks. More generally, we hope that the present work illustrates a useful paradigm for statistical inference from dependent stochastic processes: rather than estimating directly from highly dependent observations, one may first identify a probabilistic transformation that reveals an embedded conditionally independent structure and then exploit the the rich statistical methodology available for independent observations.

\medskip
\noindent\textbf{Acknowledgment.}
The authors sincerely thank Onno Boxma and Liron Ravner for their careful reading of earlier versions of this manuscript and for their helpful comments and suggestions.

\end{document}